\documentclass[reqno,12pt]{amsart}

\pdfoutput=1

\usepackage{latexsym}
\usepackage[centertags]{amsmath}
\usepackage{amsfonts}
\usepackage{amssymb}
\usepackage{amsthm}
\usepackage{newlfont}
\usepackage{graphics}
\usepackage{color}

\usepackage[usenames,dvipsnames]{xcolor}
\usepackage{graphicx}
\def\dessin#1#2{\includegraphics[#1]{#2}}
\usepackage{wrapfig}

\def\r{\mathfrak{r}}
\newtheorem{theo}{Theorem}
\newtheorem{defn}[theo]{Definition}
\newtheorem{exa}[theo]{Example}
\newtheorem{lem} [theo]{Lemma}
\newtheorem{cor}[theo]{Corollary}
\newtheorem{prop}[theo]{Proposition}
\newtheorem{rem}[theo]{Remark}

\newtheorem{conj}[theo]{Conjecture}

\newcommand{\area}{\operatorname{\texttt{area}}}

\newcommand{\dinv}{\operatorname{\texttt{dinv}}}
\newcommand{\bounce}{\operatorname{\texttt{bounce}}}

\def\sort{\operatorname{\mbox{\texttt{sort}}}}

\title[A search algorithm for inverting $\Phi$]{An efficient search algorithm for inverting the sweep map on rational Dyck paths}

\author{Guoce Xin}
\address{School of Mathematical Sciences, Capital Normal University, Beijing 100048, PR China}
\email{guoce.xin@gmail.com}

\date{May 4, 2015} 

\begin{document}

\begin{abstract}
Given a coprime pair $(m,n)$ of positive integers, 
rational $(m,n)$-Dyck paths are lattice paths in the $m\times n$ rectangle that never go below the diagonal.
The sweep map of a rational $(m,n)$-Dyck paths $D$ is the rational Dyck path $\Phi(D)$ obtained by sorting the steps of $D$ according to the ranks of their starting points, where
the rank of $(a,b)$ is $bm-an$.
It is conjectured to be a bijection, but to this date, $\Phi$ is only known to be bijective for the Fuss case ($m=kn\pm 1$). 
In this paper we give an efficient search algorithm for inverting the $\Phi$ map. Roughly speaking, given $\sigma\in \cal D_{m,n}$, by searching through a $d$-array tree of certain depth, we can output all $D$ such that $\Phi(D)=\sigma$, where
$d$ is the remainder of $m$ when divided by $n$. In particular, we show that $\Phi$ is invertible for the Fuss case  by giving a simple recursive construction for $\Phi^{-1} (\sigma)$.
\end{abstract}

\maketitle


\section{Introduction}
Throughout this paper, $(m,n)$ is always a coprime pair of positive integers, unless specified otherwise. Rational $(m,n)$-Dyck paths are lattice paths in the $m\times n$ rectangle that never go below the diagonal line $y=\frac{n}{m}x$. Denote by $\cal D_{m,n}$ the set of such paths. Many works have been done on rational Dyck paths, we are particularly interested with the
sweep map $\Phi: \ \cal D_{m,n} \mapsto \cal D_{m,n}$ that we are going to define.

An $(m,n)$-Dyck path $D$ is encoded as a $(u,d)$ (stand for going ``north" and ``east" respectively) sequence $D=p_1p_2\cdots p_{m+n}$ of $n$ $u$'s and $m$ $d$'s. Each $p_i$ is associated with a rank $r_i$ that is recursively defined by $r_1=0$, and for $2\le i\le m+n$, $r_{i}=r_{i-1}+m$ if $p_{i-1}=u$ and $r_{i}=r_{i-1}-n$ if $p_{i-1}=d$. Then the sweep map $\Phi(D)$ is obtained by sorting the $p_i$'s according to their ranks increasingly. It is true but not obvious that $\Phi(D)$ is still in $\cal D_{m,n}$. See \cite{sweepmap} for a proof.
\begin{exa}\label{exa-D-rD}
Let $(m,n)=(11,5)$.
A Dyck path $D$ and the associated rank are displayed below as a two line array.
$$  \left[\begin{array}{c} D \\ r(D) \end{array}  \right]= \left[ \begin {array}{cccccccccccccccc}
u& d& u &d &d &u &d & d& d& u& d&u  & d& d& d&d
\\0& 11& 6 &17 &12 &7 &18 & 13& 8& 3& 14&9 & 20& 15& 10&5
\end {array}
 \right].
$$
Sorting the above two line array according to the second row gives
$$  \left[\begin{array}{c} \Phi(D) \\ \sort(r(D)) \end{array} \right]= \left[ \begin {array}{cccccccccccccccc} u&u&d&u&u&d
&u&d&d&d&d&d&d&d& d & d
\\0& 3& 5 &6 &7 &8 &9 & 10& 11& 12& 13&14 & 15& 17& 18&20
\end {array}
 \right].
$$
Then the top row is $\Phi(D)$.
\end{exa}
Sweep map has become an active subject in the recent 15 years. Variations and extensions have been found, and some classical bijections turn out to be disguised sweep map. 
It first appears as the $\zeta$ map on standard Dyck paths in Haiman's study of diagonal harmonics and $q,t$-Catalan numbers \cite{Haiman}. Haglund \cite{Haglund-qtCatalan conjecture} discovered the $(\area,\bounce)$ statistic for $q,t$-Catalan numbers, and Haiman come up with the $(\dinv,\area)$ statistic. Only $\area(D)$ has a simple description: it is the number of lattice squares to the right of $D$ and to the left of the diagonal.
Sweep map provides a unified framework for the complex combinatorial structure related to $q,t$-Catalan numbers, because it takes the statistic $\dinv$ to $\area$ to $\bounce$.
However a simple description of $\bounce$ is only known in the Fuss case $m=kn\pm 1$ by Loehr \cite{Loehr-higher-qtCatalan}.
See \cite{sweepmap} for a detailed history. 

One major open problem is the invertibility of the sweep map. See \cite[Conjecture 3,3]{sweepmap} for an extended version. 
\begin{conj}
  The sweep map $\Phi$ is a bijective transformation on $\cal D_{m,n}$.
\end{conj}

The conjecture has only been proved for the Fuss cases in \cite{Loehr-higher-qtCatalan} and \cite{Gorsky-Mazin2}.
\begin{theo}\label{t-fuss}
The sweep map $\Phi$ is a bijection in the Fuss case $m=kn\pm 1$.
\end{theo}
The known inverse bijections for the Fuss case are by no means simple. They rely on the $\bounce$ statistic.

The major contribution of this paper is the ReciPhi algorithm, which is an efficient search algorithm for inverting the $\Phi$ map. 
Roughly speaking, given $\sigma\in \cal D_{m,n}$, by searching through a $d$-array tree of certain depth, we can output all $D$ such that $\Phi(D)=\sigma$, where
$d$ is the remainder of $m$ when divided by $n$. In particular, we show that $\Phi$ is invertible for the Fuss case by giving a simple recursive algorithm FussiPhi.

The discovery of the ReciPhi algorithm relies on Theorems \ref{t-rank-set} and \ref{t-EN-sequence}, 
which state that deciphering $D$ for given $\sigma=\Phi(D)$ is equivalent to deciphering either the rank sequence $r(D)$, or the EN-sequence
$\rho(D)=\Phi(D^T)$, where $D^T$ is the transpose of $D$. Theorem \ref{t-EN-sequence} was discovered  a bit earlier by Ceballos-Denton-Hanusa in \cite{zeta-map}. Indeed, their paper overlaps a lot of the author's work: Sections 5,6, and 9 overlaps this draft, and Section 4 overlaps an earlier draft of the author \cite{Xin-Rank-Complement} on $(m,n)$-cores, which is in bijection with $(m,n)$-Dyck paths by Anderson's map \cite{Anderson}. Thus it is necessary to explain what is new in this paper.

The paper is organized as follows. Section 1 is this introduction. Section \ref{sec:notation} introduces some basic notations and establishes Proposition \ref{p-Rank-Set}, which gives the fundamental characterization of the rank sets of $(m,n)$-Dyck paths. This allows rank sets to play a leading role in our development. Note that rank also appear as level (see, e.g., \cite{sweepmap} and \cite{Amstrong-Catalan}), but was never used as the basic tool for inverting the $\Phi$ map. 
Section \ref{sec:two} includes Theorems \ref{t-rank-set} and \ref{t-EN-sequence} on inverting the $\Phi$ map using rank sets and EN-sequences. The former seems new while the latter was known by Ceballos-Denton-Hanusa as we have explained. Combining the two theorems allows us to extract some partial information that are essential to sharply narrow down the search in the ReciPhi algorithm. Section \ref{sec:Left} introduces the left and right operations on rank sequences. Ceballos-Denton-Hanusa also found these operations in \cite{zeta-map} and used them to invert the $\Phi$ map. They were only able to search for $\Phi^{-1}(\sigma)$ through a much larger tree. In Lemma \ref{l-rec-left} we give more explicit formulas than that in \cite{zeta-map} for the SW-sequences and rank sequences under left operations. Section \ref{sec:ReciPhi} includes two algorithms: the FussiPhi algorithm for inverting the $\Phi$ map in the Fuss case, and the ReciPhi algorithm for inverting the $\Phi$ map by searching.


\section{Notations\label{sec:notation}}

A path $P=p_1\cdots p_{n}$ is a $(u,d)$-sequence depicted in the plane as a sequence of lattice points $P_0,\dots, P_{n}$ such that
$P_0=(0,0)$ and $P_i-P_{i-1}$ equals $(0,1)$ if $p_i=u$ and $(1,0)$ if $p_i=d$.
A $u$ step will also appears as an $S$ or $N$, standing for south or north end; Similarly a $d$ step will appear as an $E$ or $W$, standing for east or west end.
We denote by $\cal F_{m,n}$ the set of all (free) paths from $(0,0)$ to $(m,n)$. Thus $P\in \cal F_{m,n}$ consists of $n$ up steps and $m$ down steps.
An $(m,n)$-Dyck path $D$ is a path in $\cal F_{m,n}$ that never goes below the diagonal line $y=\frac{n}{m} x$. We denote by $\cal D_{m,n}$ the set of all $(m,n)$-Dyck paths.

We define the rank of a lattice point $(a,b)$ by $r(a,b)=mb-na$. Then all the lattice points $\{(a,b): 0\le a \le m, \ 0\le b \le n\}$ in the $m$ by $n$ rectangle are in bijection with their ranks except for  $r(0,0)=r(m,n)=0$. Thus we can encode a path $P \in \cal F_{m,n}$ by its rank set $R(P)=\{ r(P_0),\dots, r(P_{m+n-1}) \}= \{r(P_1),\dots, r(P_{m+n})\} $ of $m+n$ distinct ranks.
We call the sequence $r(P)=(r(P_0),\dots, r(P_{m+n-1}))$ the rank walk of $P$.
Denote by $\mathfrak{R} _{m,n}$ and $\mathfrak{R}^+_{m,n}$ the set of rank sets of $(m,n)$-paths and $(m,n)$-Dyck paths, respectively. The rank set of $(m,n)$-Dyck paths has only nonnegative ranks.

The following result completely characterizes elements of $\mathfrak{R}_{m,n}$.
\begin{prop}\label{p-Rank-Set}
Suppose $(m,n)$ is a coprime pair. Let $R=\{r_1,\dots, r_{m+n}\}$ be a set of ranks with $0\in R$. Then the following statements are equivalent.
\begin{enumerate}
\item  $R\in  \mathfrak{R}_{m,n}$, i.e., $R$ is the rank set of an $(m,n)$-path $P$.

  \item If $\ell$ belongs to $ R$ then at least one of $\ell+m$ and $\ell-n$ belongs to $R$.

  \item If $\ell\in R$ then exactly one of $\ell+m$ and $\ell-n$ belongs to $R$.

  \item If $\ell$ belongs to $ R$ then at least one of $\ell-m$ and $\ell+n$ belongs to $R$.

  \item If $\ell\in R$ then exactly one of $\ell-m$ and $\ell+n$ belongs to $R$.
\end{enumerate}
\end{prop}
\begin{proof}
Statement (1) clearly implies statement $h$ for $2\le h \le 5$.

We prove the converse implication by constructing a walk $w$ on $R$ starting with $\ell_1=0$, and simultaneously obtain the path $P$ with $R(P)=R$.
We only consider the $h=2$ case, the others are similar.
For $i\ge 2$, set $\ell_{i}= \ell_{i-1}+m$ if it belongs to $R$, and set $p_{i-1}=u$; Otherwise set
  $\ell_{i}=\ell_{i-1}-n$, which must belongs to $R$ by Condition 2 or 3, and set $p_{i-1}=d$.

  Since we are walking in the finite set $R$, we must have $\ell_{k_1}=\ell_{k_2}$ for some $1\le k_1<k_2\le m+n+1$. Assume that we walk from $\ell_{k_1}$ to $\ell_{k_2}$ by $b$ $u$ steps and $a$ $d$ steps. Then $bm-an=0$. The coprimality of $m$ and $n$ forces $a=sm,\ b=sn$ for some $s\ge 0$. Then the natural bound $0<a+b\le m+n$ implies that $a=n$ and $b=m$. Hence we obtain a closed walk of length $m+n$ from $\ell_1=0$ to $\ell_{m+n+1}=0$. Consequently, $P=p_1p_2\cdots p_{m+n}$ is the desired path with $R(P)=R$.

To see that we can not have both $\ell+m$ and $\ell-n$ belongs to $R$, suppose the walk from $\ell-n$ to $\ell+m$ takes $b$ $u$ steps and $a$ $d$ steps with $0<a+b< m+n$. Then
 $bm-an=m+n$, i.e., $(b-1)m-(a+1)n=0$. This implies that $b-1=sn$ and $a+1=sm$ for some $s\ge 0$, which is clearly impossible.
\end{proof}
\def\rc{\mathfrak{c}}

A direct consequence of Proposition \ref{p-Rank-Set} is the following.
\begin{cor}
Let $R$ be a rank set. Then i) $R\in \mathfrak{R}^+_{m,n}$ if and only if $R\in \mathfrak{R}^+_{n,m}$; ii) $R\in \mathfrak{R}^+_{m,n}$ if and only if
the rank complement $\rc (R)= (\max R) -R$ is in $\mathfrak{R}^+_{m,n}$.
\end{cor}
We use $R^T$ to indicate that we switch the roles of $m$ and $n$, i.e,, $R^T=R$ is a rank set in $\mathfrak{R}^+{n,m}$. This corresponds to the transpose of Dyck paths.
See \cite{Xin-Rank-Complement} for geometric meaning of rank complement.

From now on, we shall identify a Dyck path $D$ with its rank set $R=R(D)$. So a Dyck path $R$ means the Dyck path $D$ with rank set $R(D)=R$, and a rank set $D$ means $R(D)$.
The reason for making this convention is that many operations are easy to perform on rank sets.

Denote by $S(D)$ the set of south ends of $D$, i.e, starting points of a $u$ step of $D$. In our convention, $S(D)$ is identified with $S(R)=\{ s_0,s_1,\dots, s_{n-1} \}$, which is the set of the ranks of the south ends of $D$. In Example \ref{exa-D-rD}, $S(D)=\{0,3,6,7,9 \}$. An easy exercise is ${ S(D)} \bmod{n} =\{0,1,\dots, n-1\}$.
We can similarly define $E(D),N(D),W(D)$ to be the ranks of the east ends, north ends, and west ends of $D$ respectively. We have the following equalities:
\begin{align}
S(D)& =N(D)-m = R(D) \cap (R(D)-m), \label{e-SN} \\
W(D)&= E(D)+n= R(D) \cap (R(D)+n),\label{e-WE}\\
  R(D)&=S(D)\uplus W(D)= E(D)\uplus N(D), \label{e-SW}
\end{align}
where $\uplus$ means disjoint union. Equations in \eqref{e-SW} follow from the fact that each node of $D$ is either a south end or a west end, and similarly is either an east end or a north end. Note that $0$ is a south end as the node $(0,0)$ and an east end as the node $(m,n)$.

Denote the sorted increasing sequence of $R=R(D)$ by
$$\r(D)=\sort(R)= (\r_1(R), \r_2(R),\dots, \r_{m+n}(R))= (\r_1,\dots, \r_{m+n}), \ \text{with } \r_1=0.$$
The SW-sequence of $D$ is defined by
$$ \sigma (D) =\sigma(R)=(\sigma_1,\sigma_2, \dots, \sigma_{m+n}),$$
where $\sigma_i$ is $S$ if $\r_i\in S(D)$ and is $W$ if $\r_i\in W(D)$.
Then the sweep map $\Phi(D)$ is just $\Phi(D)= \sigma (D) \big|_{S=u,W=d}$. We use the notation $\sigma\in \cal D_{m,n}$ for
$\sigma\big|_{S=u,W=d} \in \cal D_{m,n}$.

Similarly, the EN-sequence of $D$ is
defined by
$$ \rho (D) =\rho(R)=(\rho_1,\rho_2, \dots, \rho_{m+n}),$$
where $\rho_i$ is $E$ if $\r_i\in E(D)$ and is $N$ if $\r_i\in N(D)$. See Figure \ref{fig:Dyckpath}. 

\begin{figure}[ht]
\begin{center}
\begin{displaymath}
\mbox{\dessin{width=420pt}{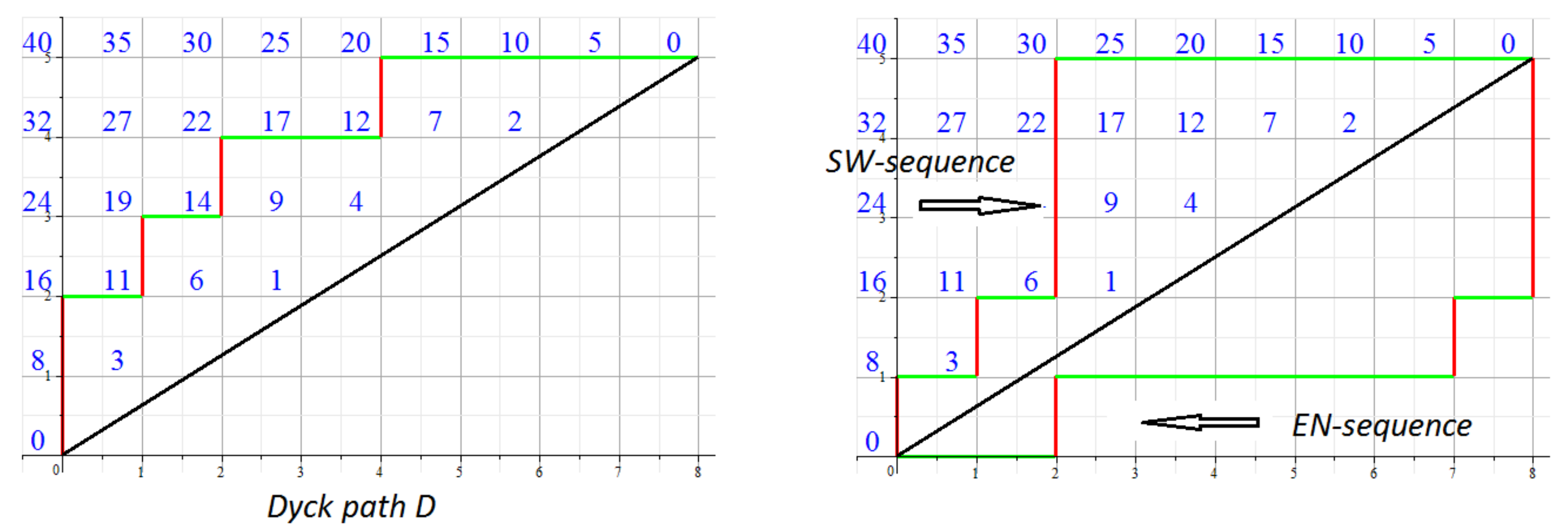}}
\end{displaymath}
\caption{A Dyck path $D$ with $\area(D)= 7$, its SW-sequence, and its EN-sequence (put below the diagonal).}
\label{fig:Dyckpath}
\end{center}
\end{figure}
\def\rev{\texttt{rev}}

The SW-sequence and EN-sequence are closely related by $\rho(D)=\sigma(D^T)|_{S=N,E=W}$, since transpose keeps the rank set, but takes south ends to east ends, and west ends to north ends. A nicer description is the following.
\begin{lem}\label{l-chi}
For $R\in \mathfrak{R}^+_{m,n}$, we have
$$\rev (\rho(R))=(\rho_{m+n},\dots, \rho_1)= \sigma(\rc ( R))|_{S=N,W=E} .$$
\end{lem}
\begin{proof}
Let $\r=\sort(R)=(\r_1,\dots,\r_{m+n})$ and $M=\max R$. Then
$$ \sort(\rc ( R))= (M-\r_{m+n},\dots, M-\r_0).$$
Since rank complement takes south ends to north ends, and west ends to east ends, and vice vasa, the lemma follows.
\end{proof}

\section{Rank sequence and EN-sequence for inverting the $\Phi$ map\label{sec:two}}
For an SW-sequence $\sigma$, it is convenient to label, from left to right, the $i$-th $S$ by $S_i$ for $1\le i \le n$ and the $j$-th $W$ by $W_j$ for $1\le j \le m$.
The rank $\r(S_i)$ refers to the $i$-th largest element in $S(R)$. In formula, we have $\r(S_i)=\r_{\sigma^{-1}(S_i)}$, where $\sigma^{-1}(S_i)$ is the position $i'$ of $S_i$, i.e., $\sigma_{i'}=S_i$. Similar notations apply for $W_j$, and $E_j, N_i$ for the EN-sequence.

For the Dyck path in Figure \ref{fig:Dyckpath}, we have
$$ \left[ \begin {array}{l|ccccccccccccc}\text{position} &1&2&3&4&5&6&7&8&9&10&11&12&13
\\ 
\sigma(D) & S_{{1}}&W_{{1}}&S_{{2}}&W_{{2}}&S_{{3}}&S_{{4}}&S
_{{5}}&W_{{3}}&W_{{4}}&W_{{5}}&W_{{6}}&W_{{7}}&W_{{8}}\\
\r(D) &0&5&8&10&11&12&14&15&16&17&19&20&22
\\
\rho(D) &E_{{1}}&E_{{2}}&N_{{1}}&E_{{3}}&E_{{4}}&E_{{5}}&E
_{{6}}&E_{{7}}&N_{{2}}&E_{{8}}&N_{{3}}&N_{{4}}&N_{{5}}\end {array} \right]
  $$

\begin{defn}
Given an SW-sequence $\sigma=(\sigma_1,\dots,\sigma_{m+n})$, a nonnegative rank set $R\ni 0$ or $\r=\sort(R)$ is said to be compatible with  $\sigma$ if
$\r(S_i)+m\in R$ for $1\le i\le n$ and $\r(W_i)-n\in R$ for $1\le j\le m$.
\end{defn}

\begin{theo}\label{t-rank-set}
An SW-sequence $\sigma=(\sigma_1,\dots,\sigma_{m+n})$ is equal to $\sigma(D)$ for some $(m,n)$-Dyck path $D$ if and only if there is a nonnegative rank set $R$ compatible with $\sigma$.
\end{theo}
\begin{proof}
The sufficiency is by equations \eqref{e-SN} and \eqref{e-WE}. For the necessity, if $\r=\sort(R)$ is compatible with $\sigma$, then each rank $r$ is either $r=\r(S_i)$ then $r+m\in R$
 or $r=\r(W_j)$ then $r-n\in R$. Thus by Proposition \ref{p-Rank-Set}, $R \in \cal R^+_{m,n}$ and $\r(D)=\r$ for some Dyck path $D$. The lemma then holds since $\sigma(D)=\sigma(R)=\sigma$.
\end{proof}

\begin{exa}\label{exa-R0}
The rank set $R_0=\{0,1,\dots, m+n-1\}$ corresponds to the unique Dyck path of area $0$.
Since $i+m\in R_0$ for $0\le i\le n-1$, $\sigma(R_0)=S^n W^m$, i.e., $n$ $S$'s followed by $m$ $W$'s. Conversely, if we are given
$\sigma=S^nW^m$, we can deduce that $\r(W_1)=n+\r(E_1)=n$. Now $\r(S_i)<n$ for all $i$ and they are distinct. Thus $\r(S_i)=i-1$ for all $i$, in particular
$\r(N_n)=\r(S_n)+m=n+m-1$ is $\max(\r)$. It follows that $\r=\sort(R_0)$.
\end{exa}

\begin{defn}
For an SW-sequence $\sigma=(\sigma_1,\dots,\sigma_{m+n})\in \cal D_{m,n}$ and an EN-sequence $\rho=(\rho_1,\dots, \rho_{m+n})\in \cal D_{m,n} $. Define
$G(\sigma,\rho)$ to be the directed graph on the vertices $\{1,2,\dots, m+n\}$, with edges $\sigma^{-1}(S_i)\to \rho^{-1}(N_i), \ 1\le i\le n$ and $\rho^{-1}(E_j) \to \sigma^{-1}(W_j),\ 1\le j\le m$. If $G(\sigma,\rho)$ is a directed cycle of length $m+n$, and produces an increasing rank sequence $\r$ by the three rules i) $\r_1=0$,
ii) $\r(S_i)=\r(N_i)-m$, for $1\le i\le n$;   iii) $\r(W_j)=\r(E_j)+n$ for $1\le j\le m$, then $\rho$ is said to be compatible with $\sigma$.
\end{defn}

\begin{theo}\label{t-EN-sequence}
An SW-sequence $\sigma=(\sigma_1,\dots,\sigma_{m+n})$ is equal to $\sigma(D)$ for some $(m,n)$-Dyck path $D$ if and only if there is an EN-sequence $\rho$ compatible with 
$\sigma$. 
\end{theo}
The prove is similar to that of Theorem \ref{t-rank-set}.

Though either $\r$ or $\rho$ are sufficient to invert the $\Phi$ map, a better strategy is combining Theorems \ref{t-rank-set} and \ref{t-EN-sequence}. This is better illustrated by an example. 
\begin{exa}\label{exa-two}
Let $(m,n)=(11,5)$ and $\sigma$ give as follows. Assume Lemma \ref{l-position-kn} below gives the underlined ranks.
$$  \left[\begin{array}{c}\sigma \\ \r \end{array}\right]= \left[ \begin {array}{cccccccccccccccc} S_{{1}}&S_{{2}}&W_{{1}}&S_{{3}}&S_{{4}}&W_{{2}}
&S_{{5}}&W_{{3}}&W_{{4}}&W_{{5}}&W_{{6}}&W_{{7}}&W_{{8}}&W_{{9}}& W_{10} & W_{11}
\\ \underline{0}&3 &\underline{5} & \mathbf{6}& \mathbf{7}& \mathbf{8} & \mathbf{9}&\underline{10} &11 &12 & 13& 14&\underline{15} &17 &18 &20
\end {array}
 \right],
$$
Since there are 4 ranks between $5$ and $10$ and $\r$ is increasing, the boldfaced ranks can be filled in. Similar reasons for the ranks between $10$ and $15$.
The rank $\r(W_2)=8$ implies $\r(E_2)=3$, which has to be $\r(S_2)$. Finally $\r(N_3)=\r(S_3)+11=17$, and similarly $\r(N_4)=18$ and $\r(N_5)=20$ must appear as the final three ranks.
This rank sequence appears in Example \ref{exa-D-rD}.
\end{exa}

In general we can deduce some partial information of the EN-sequence and extract some ranks. These information are crucial for the ReciPhi algorithm.
\begin{lem}\label{l-position-kn}
Let $m=kn+d$ be coprime to $n$ with $1\le d \le n-1$ and $k\ge 1$. Given an SW-sequence $\sigma\in \cal D_{m,n}$, if there is a rank set $R$ with $\r=\sort(R)$ compatible with $\sigma$, then $0, n, \dots, kn+n$ are in $R$ and $\r_{f_i}=in$ for $0\le i\le k+1$, where $f_i$ is determined by:
$$f_0=1;   \quad  \ f_{i}=\sigma^{-1}(W_{f_{i-1}}), 1 \le i \le k+1.$$
\end{lem}
\begin{proof}
Suppose the SW-sequence of $R$ is $\sigma=\sigma(R)$.
Since $\r_1=0=\min R$, we have $\r(S_1)=\r(E_1)=0$. It follows that $\r(N_1)=\r(S_1)+m=m$ and $\r(W_1)=\r(E_1)+n=n$. The position of $W_1$ can be read from $\sigma$ but the position of $N_1$ is unknown.

Observe that all the ranks smaller than $\r(N_1)=m$ must belong to $E(R)$. Assume inductively that for $i<k+1$,  $W_{f_{i-1}}$ has rank $(i-1)n<m$. Then it is also $E_{f_i}$ and thus
$\r(W_{f_i}) =\r(E_{f_i})+n=in,$ as desired.
\end{proof}

\begin{exa} Let $(m,n)=(13,4)$. We illustrate Lemma \ref{l-position-kn} by the following example.
$$ \left[ \begin {array}{ccccccccccccccccc} 1&2&3&4&5&6&7&8&9&10&11&12&
13&14&15&16&17\\ \noalign{\medskip}S_{{1}}&S_{{2}}&W_{{1}}&W_{{2}}&S_{
{3}}&W_{{3}}&S_{{4}}&W_{{4}}&W_{{5}}&W_{{6}}&W_{{7}}&W_{{8}}&W_{{9}}&W
_{{10}}&W_{{11}}&W_{{12}}&W_{{13}}\\ 0& &4& & &8& &
 & &12& & & &16& & &\\ E_{{1}}&E_{{2}}&E_{{
3}}&E_{{4}}&E_{{5}}&E_{{6}}&E_{{7}}&E_{{8}}&E_{{9}}&E_{{10}}& &
& & & & & \end {array} \right].
 $$
From $\r_1=0=\r(E_1)$ we deduce that $\r_3=\r(W_1)=\r(E_1)+n=4<m=13$. Then we can label up to $E_3$ and have $\r(E_3)=4$.
It follows that $\r_6=\r(W_3)=\r(E_3)+4=8<13$. Thus we can label up to $E_6$ and have $\r(E_6)=8$. In a similar manner,
we have $\r_{10}=\r(W_6)=\r(E_6)+4=12<13$, which must be $\r(E_{10})=12$; and $\r_{14}=\r(W_{10})=\r(E_{10})+4=16>13$.
\end{exa}

\begin{rem}\label{rem-kn-pm}
Observe that $f_{i+1}-f_i$ is the number of $S$'s up to position $f_i$, so we need not write down the EN-sequence and the position sequence. In the running example,
we start with $\r_1=0$, $\r(W_1)=4$, the next is $\r(W_{1+2})=\r(W_3)=8$, since there are two $S$ before $W_1$; the next one is $W_{3+3}=W_6$ with rank $12$; the final one is
$W_{6+4}=W_{10}$ with rank $16$.
\end{rem}

In this example, we can actually determine the position of $m+n=17$:
$|R \cap \{0,1,\dots, m+n\}|=14$, since $\r_{14}=16$ and $\r$ is increasing. Indeed we have the following result, which is sufficient for inverting $\Phi$ in the next section.
\begin{theo}\label{t-kn-pm}
Follow notations in Lemma \ref{l-position-kn}. If $m=kn+1$, then $\r_{f_{k}+1}=m$ and $|R\cap   \{0,1,\dots, m+n\}|=f_{k+1}$;
If $m=kn+n-1$, where $k$ is allowed to be $0$, then $\r_{f_{k+1}-1}=m$, and
\begin{align}
  \label{e-kn-1}
  |R\cap   \{0,1,\dots, m+n\}|=\sigma^{-1}(W_{f_{k+1}-1})-1.
\end{align}
\end{theo}
\begin{proof}
By Lemma \ref{l-position-kn} and the fact that $\r$ is increasing, the case $m=kn+1$ is obvious and the first part of the case $m=kn+n-1$ is also obvious. Now we show the second equality when $m=kn+n-1$.

We have $\sigma^{-1}(N_1)=f_{k+1}-1$. There will be no difficulty if $W_{f_{k}}$ with rank $m+1$ is also $E_{f_{k+1}-1}$, since then $\r(W_{f_{k+1}-1})=\r(E_{f_{k+1}-1})+n=m+n+1$, which forces $|R\cap   \{0,1,\dots, m+n\}|= \sigma^{-1}(W_{f_{k+1}-1})-1,$ as desired. The situation is similar if there are several $N$ labels before $E_{f_{k+1}-1}$. Note that
\eqref{e-kn-1} is independent of the number of these $N$ labels.

Claim: If $f_{k+1}-1, \dots, f_{k+1}+s-2$ are positions of $N_1,N_2,\dots, N_s$ in $\rho(R)$ for some $s< n$, then
$\r_{f_{k+1}+j-1}=m+j$ for $j=0,1,\dots, s$.

We prove the claim by induction on $s$. The case $s=1$ is trivial. Assume the claim holds for $s-1$. Then
$\r(N_{j+1})=\r_{f_{k+1}+j-1}=m+j$ for $j=0,1,\dots, s-1$. It follows that $\r(S_{j+1})=\r(N_{j+1})-m=j$. A similar argument as in the proof of Lemma \ref{l-position-kn} shows that $\r_{f_i+j}=\r_{f_i}+j=in+j<m$ for $i\le k+1$ and $j\le s$. In particular $\r_{f_{k+1}+s}=m+s+1$.

Now suppose $N_s$ is followed by $E_{f_{k+1}-1}$. By the claim $\r(E_{f_{k+1}-1})=m+s$, 
we have  $\r(W_{f_{k+1}-1})=\r(E_{f_{k+1}-1})+n=m+n+s$.
Now for $1\le j\le s$,  $m+j-1\in N(R)$ implies that $m+j-1-m=j-1\in S(R)$,
so by Proposition \ref{p-Rank-Set} part (4), $n+m+j-1\not \in R$. Therefore $\r(W_{f_{k+1}-1})$ is the smallest rank larger than $m+n$. Equation \ref{e-kn-1} then follows.
\end{proof}

\section{Left and Right operation of the rank sequence \label{sec:Left}}
For $R_0\ne R\in \mathfrak{R}^+_{m,n}$, we can perform two operations on $R$ by:
\begin{align}
\cal L(R)&= (R\setminus \{0\} \cup \{m+n\})-\min (R \setminus \{0\}); \\
\cal R(R)&= R \setminus \{\max R\} \cup \{ \max R-m-n\};
\end{align}

The following result is an easy exercise.
\begin{lem}
  For $R_0\ne R\in  \mathfrak{R}^+_{m,n}$, $\cal L(R)$ and $\cal R(R)$ are both in $ \mathfrak{R}^+_{m,n}$.
\end{lem}

The geometric meaning of the Dyck path $\cal R(R)$ is obtained by changing the highest rank point $(a,b)$ of $D$ to $(a+1,b-1)$. This corresponds to removing a square
from the corresponding Dyck path $D$, so 
$$\area (D)= \area (\cal R(D))+1.$$ 
Repeatedly applying the $\cal R$ operation to $R$ will finally give $R_0$. So we can obtain by induction that
$$ \area (D)= \frac{1}{m+n}\left( \sum_{i=1}^{m+n} \r_i -\binom{m+n}{2}  \right).$$

The geometric meaning of $\cal L(D)$ is similar to $\cal R(D)$, but work on the rank complement of $D$. They are related by
$\cal L(D) =  \rc \circ \cal R \circ \rc ( D).$ It follows that  
$$  \area (\rc( \cal L( D ))) =\area (\rc\circ  \cal L \circ \rc \circ \rc( D )= \area(\cal R( \rc(D)))= \area(\rc(D)) -1.
$$
Iterating this formula gives
\begin{align}
  \label{e-L-area}
\area( \rc\circ \cal L^k (D))= \area(\rc(D))-k  .
\end{align}

\def\key{\texttt{key}}

The position $\delta=\delta(D)=|R(D)\cap \{0,1,\dots,m+n\}|$ was used in \cite{zeta-map}. We also use $\delta$ but we find it better to use the key position 
defined by $\key (D)=\key(R)=|S(D) \cap \{ 0,\dots, m+n \}|$ by observing that $\delta$ is between the $\key (D)$-th $S$ and the $\key(D)+1$-st $S$.  We have the following result.
\begin{lem}\label{l-rec-left}
Suppose $\sigma$ is the SW-sequence of $R\in \mathfrak{R}_{m,n}$ and $\sigma^L$ is the SW-sequence of $\cal L(R)$.
Let $s=\sigma^{-1}(W_1)$ and $\delta=\delta(R)$. Then
\begin{enumerate}
  \item[i)] $R$ can be recovered from the pair $(\cal L(R),\sigma)$ by $R= (\r_2+\cal L(R))\setminus \{m+n\} \cup \{0\}$, where
  $n-\r_2$ is equal to the $s-1$-st smallest element of $\cal L(R)$.

\item[ii)] $\sigma^L$ is obtained from $\sigma$ by inserting a $W$ after $S_{\key(R)}$, same as after $\sigma_\delta$,  and then removing $\sigma_s=W$.

\item[iii)] $\sigma$ is obtained from $\sigma^L$ by removing $\sigma^L_\delta=W$ and and then inserting a $W$ after $\sigma_{s-1}$.
\end{enumerate}
\end{lem}
\begin{proof}
Let $\r =\sort(R)$, $\cal L( \r)= \sort (\cal L(R))$. By assumption, $\sigma_s=W_1$ and $\r_\delta<m+n<\r_{\delta +1}$, then
$$\r^L=\cal L( \r)= (0=\r_2-\r_2,\r_3-\r_2,\dots, \r_{\delta}-\r_2, m+n-\r_2, \r_{\delta +1}-\r_2,\dots, \r_{m+n}-\r_2).$$
Clearly, we have $\r^L_{s-1}=\r_s-\r_2=n-\r_2$.

i) $\r$ can be reconstructed from the triple $(\r^L,s,t)$ by the formula:
\begin{align}
  \label{e-r-by-rL}
\r = (0, \r^L_1+\r_2,\dots, \r^L_{\delta -1}+\r_2, \r^L_{\delta +1}+\r_2,\dots, \r^L_{m+n}+\r_2),   \text{ where } \r_2=n-\r^L_{s-1}.
\end{align}
But with $(\r^L, \sigma)$, we have: $s=\sigma^{-1}(W)$ and $\delta$ is the position of $\r^L_{s-1}+m$ in $\r^L$.

ii) We discuss by cases: a) If $i>1$ and $\r_i \in S(R)$ then $\r_i+m=\r_j$ for some $j>i$. It follows that $\r_i-\r_2+m=\r_j-\r_2$, so that $\r_i-\r_2\in S(\cal L(R))$;
b) If $j\ne s$ and $\r_j \in W(R)$ then $\r_j-n=\r_i$ for some $1\ne i<j$. It follows that $\r_j-\r_2-n=\r_i-\r_2$, so that $\r_j-\r_2\in W(\cal L(R))$;
c) $\r_s-\r_2 \in S(\cal L(R))$ since $\r_s-\r_2+m=m+n-\r_2\in r(\cal L(R))$;
d) $m+n -\r_2 \in W(\cal L(R))$ since $m+n-r_2-n=m-r_2 \in r(\cal L(R))$.

Therefore, the SW-sequence $\sigma^L$ is obtained from $\sigma$ by inserting a $W$ after $\sigma_\delta$, changing $\sigma_s=W$ to $S$, and finally removing $\sigma_1$. This is equivalent to the description in the lemma, since $\sigma_i=S$ for $i=1,2,\dots, s-1$.

iii) Follows directly from part ii).
\end{proof}

\begin{cor}\label{c-Phi-invert}
If for every $D\in \cal D_{m,n}$, the number $\key(D)$ can be uniquely determined by $\sigma(D)$, then the $\Phi$ map is invertible.
\end{cor}
\begin{proof}
We prove the injectivity of $\Phi$ by induction on $\area(\rc(D))$. The bijectivity then follows since $\Phi$ is a map from $\cal D_{m,n}$ to itself.

The base case is $\sigma(R_0)=S^n W^m$, as discussed in Example \ref{exa-R0}. 
If there is a $D$ such that $\sigma(D)=\sigma$, then $\area (\rc(\cal L(D)))=\area (\rc (D))-1 $. By induction $\r^L$ can be constructed from $\sigma^L$, which can be constructed 
from $\sigma$ and $\key(D)$ by 
Lemma \ref{l-rec-left}. Then $R$ can be constructed from $(\r^L,\sigma)$, again by Lemma \ref{l-rec-left}.
\end{proof}

From the above proof, we see a possible approach for inverting the sweep map by searching all possible values of $\delta$. The correct $\delta$ will gives rise the desired rank set $R$. Because we need to guess in each of the $\area( \rc(D))$ recursive steps, we are indeed searching (roughly) through an $(m+n)$-ary tree of depth $\area( \rc(D))$. Using $\key(D)$ is better, since there are only $n$ possible value of $\key(D)$, corresponding to the $n$ $S$'s. Searching through an $n$-ary tree is still not a good strategy. Using Lemma \ref{l-position-kn} may sharply narrow down the searchs.

\section{The ReciPhi Algorithm for inverting the sweep map \label{sec:ReciPhi}}

We start by combining Corollary \ref{c-Phi-invert} and Theorem \ref{t-kn-pm} to recover Theorem \ref{t-fuss} for the Fuss case $m=km\pm 1$.
Indeed, we have a simple recursive algorithm as follows.

The FussiPhi Algorithm.

Input: A coprime pair $(m,n)$ with $m=kn\pm 1, \ k\ge 1$; An SW-sequence $\sigma\in \cal D_{m,n}$.

Output: A rank sequence $\r$ compatible with $\sigma$.

\begin{enumerate}
  \item If $\sigma=S^n W^m$,  then output $\r=(0,1,\dots, m+n-1)$.

  \item Find by Lemma \ref{l-position-kn} the position $s$ of $n$ and position $\delta$ of $m+n$, i.e., $\r_s=n, \ \r_\delta<m+n<\r_{\delta+1}$.

  \item Let $\sigma^L$ be obtained from $\sigma $ by inserting a $W$ after $\sigma_\delta$ and then removing $\sigma_s=W$.
 Recursively construct $\r^L=\r(\sigma^L)$. Output $\r$ using Lemma \ref{l-rec-left} for $(\r^L,\sigma)$.
\end{enumerate}

\begin{exa}
We illustrate the FussiPhi algorithm by redoing Example \ref{exa-two}. Recall that $(m,n)=(11,5)$.
$$  \left[\begin{array}{c} \sigma \\ \r \end{array}\right]= \left[ \begin {array}{cccccccccccccccc} S_{{1}}&S_{{2}}&W_{{1}}&S_{{3}}&S_{{4}}&W_{{2}}
&S_{{5}}&W_{{3}}&W_{{4}}&W_{{5}}&W_{{6}}&W_{{7}}&W_{{8}}&W_{{9}}& W_{10} & W_{11}
\\ 0& &5 & & & & &10 & & & & &15 & & &
\end {array}
 \right],
$$
where the partial ranks are computed by Remark \ref{rem-kn-pm}. Thus $s=3$, $\delta=\sigma^{-1}(W_8)$, so
$$  \left[\begin{array}{c} \sigma' \\ \r'\end{array} \right]= \left[ \begin {array}{cccccccccccccccc} S_{{1}}&S_{{2}}&S_{{3}}&S_{{4}}&W_{{1}}
&S_{{5}}&W_{{2}}&W_{{3}}&W_{{4}}&W_{{5}}&W_{{6}}&W_{{7}}&W_{{8}}&W_{{9}}& W_{10} & W_{11}
\\ 0& & & &5 & & & & &10 & & & & &15 &
\end {array}
 \right].
$$
Similarly $s'=5$, and $\delta'=\sigma^{-1}(W_{10})$. Thus
$$  \left[\begin{array}{c} \sigma'' \\ \r''\end{array} \right]= \left[ \begin {array}{cccccccccccccccc} S_{{1}}&S_{{2}}&S_{{3}}&S_{{4}}
&S_{{5}}&W_{{1}}&W_{{2}}&W_{{3}}&W_{{4}}&W_{{5}}&W_{{6}}&W_{{7}}&W_{{8}}&W_{{9}}& W_{10} & W_{11}
\\ 0&1 &2 &\underline{3} &4 & 5& 6& 7& 8&9 & 10& 11& 12& 13&\underline{14} &15
\end {array}
 \right].
$$
Now adding all the ranks by $5-3=2$, removing $m+n$ and adding $0$ gives back
$$  \left[\begin{array}{c} \sigma' \\ \r'\end{array} \right]= \left[ \begin {array}{cccccccccccccccc} S_{{1}}&S_{{2}}&S_{{3}}&S_{{4}}&W_{{1}}
&S_{{5}}&W_{{2}}&W_{{3}}&W_{{4}}&W_{{5}}&W_{{6}}&W_{{7}}&W_{{8}}&W_{{9}}& W_{10} & W_{11}
 \\0& \underline{2} &3 &4 &5 &6 & 7& 8& 9& 10&11 & 12& \underline{13}& 14& 15&17
\end {array}
 \right].
$$
Now adding all the ranks by $5-2=3$, removing $m+n$ and adding $0$ gives back
$$  \left[\begin{array}{c} \sigma \\ \r\end{array} \right]= \left[ \begin {array}{cccccccccccccccc} S_{{1}}&S_{{2}}&W_{{1}}&S_{{3}}&S_{{4}}&W_{{2}}
&S_{{5}}&W_{{3}}&W_{{4}}&W_{{5}}&W_{{6}}&W_{{7}}&W_{{8}}&W_{{9}}& W_{10} & W_{11}
\\0& 3& \underline{5} &6 &7 &8 &9 & 10& 11& 12& 13&14 & 15& 17& 18&20
\end {array}
 \right].
$$
\end{exa}

Now we are ready to give:

The ReciPhi Algorithm for $m=kn+d, \ k\ge 1, 1\le d<n$.

Input: A coprime pair $(m,n)$ as above ; An SW-sequence $\sigma\in \cal D_{m,n}$.

Output: Determine the existence of the rank sequence $\r$ compatible with $\sigma$. Output $\r$ if exists.

\begin{enumerate}
  \item If $\sigma=S^n W^m$,   then output $\r=(0,1,\dots, m+n-1)$.

  \item Find by Lemma \ref{l-position-kn} the position $s$ of $n$ and position $a$ of $kn+n$, i.e., $\r_s=n, \ \r_a=kn+n$. Let $C=\{ c :  a< \sigma^{-1}(S_{c+1}) \text{ and } \sigma^{-1}(S_c)<a+d\}$, where we set $\sigma^{-1}(S_{n+1})=m+n+1$.

  \item For each $c\in C$ as a candidate of $\key(R)$, define the interval $I_c$ to be $I_c=[\sigma^{-1}(S_c), \sigma^{-1}(S_{c+1})-1]$. Let $\sigma'$ be obtained from $\sigma $ by inserting a $W$ after $S_c$ and then removing $\sigma_s=W$.
  If $\sigma'\ne \sigma$ then recursively construct $\r'=\r(\sigma')$. Find the position $\delta$ so that $\r'_\delta= \r'_{s-1}+m$. If $\sigma'_\delta=W$ and $\delta \in I_c$ then
  $\sigma'$ is the desired $\sigma^L$. Output $\r$ using Lemma \ref{l-rec-left}.

\item If no $\r$ can be found in Step 3, then there is no $\r $ compatible with $\sigma$.
\end{enumerate}

\begin{proof}[Justification of the ReciPhi Algorithm]
Assume the existence of $\r$. Step 1 is trivial.

Step 2 is to determine the possible position $\delta$ of $m+n$, i.e., $\r_\delta <m+n<\r_{\delta+1}$. Since $\r_a=kn+n$ and $\r$ is a strictly increasing sequence, $\r_{a+d}\ge \r_a +d=kn+n+d=m+n$, where the equality can not hold since $m+n \not\in \r$. It follows that
the $\delta$  must satisfy the condition $a\le \delta \le a+d-1$. That is, $\delta \in I_c$ for some $c\in C$.

In Step 3, for each candidate $c\in C$ of $\key(R)$, we try to construct $\r$ using $\delta \in I_c$. The $\sigma'$ is just the $\sigma^L$ constructed by Lemma \ref{l-rec-left} from $(\sigma,s,\delta)$, which are  same for all $\delta \in I_c$. Lemma \ref{l-rec-left} also applies to give the position of $\delta $ by $\r'_t=\r'_{s-1}+m$. If this $\delta $ matches the condition $\delta \in I_c$, then Lemma \ref{l-rec-left} says that
 the $\r$ constructed from $(\r',\sigma)$ is compatible with $\sigma$.

In Step 4, since all candidates fail, there is no $\r$ compatible with $\sigma$.
\end{proof}

The complexity of the ReciPhi Algorithm:

If $\sigma(D)=\sigma$, then the search tree has depth $\area(\rc( D))$. Each node has at most $d$ children corresponding to the choices of $\delta$. In the worst case, we need to
search a $d$-array tree of depth $\area(\rc( D))$. In practice, the algorithm uses $\key(D)$ instead of $\delta$, which may significantly reduce the number of children. The Maple implementation of this algorithm is pretty fast when $d$ is small.

Examples for illustrating the ReciPhi algorithm is too lengthy to be presented here, so we omit it. 

\section{Concluding remark}
The $\Phi$ map only takes $O(m+n)^2$ time, mostly for sorting. It is hard to believe that inverting it is so hard. 
Can we find a polynomial time algorithm for inverting $\Phi$? At least we have presented a polynomial time algorithm 
for the Fuss cases. 

The author believes that the rank set should play the leading role for inverting the sweep map. The walk on the rank set $R$ for coprime $(m,n)$ is simply
a closed circle of length $m+n$. If $m$ and $n$ are not coprime, then the walk will be a Eulerian tour, as described in \cite{Loehr-Warrington} in the study of many statistics on partitions. The Eulerian tour was used to study the $h$-statistic, which is just the $\dinv$ statistic when restricted to $(m,n)$-Dyck paths.

Finding the EN-sequence $\rho$ is a good strategy for inverting $\Phi$. Since $\rev(\rho)|_{E=d,N=u}\in \cal D_{m,n}$, it is natural to define the $\chi$ map on $\cal D_{m,n}$ by $\chi(\sigma)=\rev(\rho)$. Lemma \ref{l-chi} is equivalent to saying that $\chi = \Phi \circ \rc \circ \Phi^{-1}$ when assuming $\Phi$ is invertible. Since $\rc$ is an involution, so is the $\chi$ map. Indeed $\chi$ is an area preserving involution \cite[Proposition 6.9]{zeta-map}.
The EN-sequence for the $m=n+1$ case is particularly nice: 
Removing the final $W$ in $\sigma$, and removing the first $E$ in $\rho$, one will see the similarity of the resulting sequence. For example:
If $(m,n)=(5,4)$ and $R=(0, 4, 5, 8, 10, 11, 12, 15, 16)$, then the shifted SW- and EN-sequence is given by
$$ \left[ \begin {array}{cccccccccc} & S&W&S&W&S&S&W&W&W
\\ E&E&N&E&N&E&E&N&N &\end {array} \right].$$
This pattern is always true. See \cite[Theorem 7.2]{zeta-map}. It is natural to hope for a similar result for the Fuss case. 
We discovered the $\chi$ map independently, and found a desired solution for the Fuss case. This will appear in an upcoming paper \cite{Garsia-Xin-Fuss}.

The main objective of this paper is to introduce the ReciPhi algorithm, which seems to be the first efficient algorithm inverting the $\Phi$ map. However, the algorithm does not perform well when $d$ is large, so different technique should be developed for large $d$. A similar argument as for $m=kn-1$ might work when $n-d$ is small. 

It is still hard to prove that $\Phi$ is a bijection even when $m=n+2$ with $n$ odd. The major problem is that the recursion requires some property of $\cal L(R)$, which is little known.

\medskip
\noindent
{\small \textbf{Acknowledgements:} This work was done during the author's visiting at UCSD. The author is very grateful to Professor Adriano Garsia for
inspirations and encouraging conversations.
%
This work was partially supported by the National Natural Science Foundation of China (11171231).}


\begin{thebibliography}{99}
\bibitem{Anderson}
Jaclyn Anderson, Partitions which are simultaneously $t_1$- and $t_2$-core, Discrete Math., 248 (1-3) (2002) 237--243.


\bibitem{sweepmap}
 Drew Armstrong, Nicholas A. Loehr, and Gregory S. Warrington, Sweep maps: A continuous family of sorting
algorithms, preprint, arXiv:1406.1196.

\bibitem{Amstrong-Catalan}
Drew Armstrong, Nicholas A. Loehr, and Gregory S. Warrington, Rational parking functions and Catalan
numbers, preprint, arXiv:1403.1845.


\bibitem{zeta-map}
Cesar Ceballos, Tom Denton, and Christopher R. H. Hanusa, Combinatorics of the zeta map on rational Dyck paths, preprint, arXiv:1504.06383.

\bibitem{Garsia-Xin-Fuss}
Adriano Garsia and Guoce Xin, A simple inverse bijection for sweep map in the Fuss case, in preparation.  

\bibitem{Haglund-qtCatalan conjecture}
James Haglund, Conjectured statistics for the $q; t$-Catalan numbers, Adv. in Math., 175 (2003), 319--334.

\bibitem{Haiman}
Mark Haiman, Conjectures on the quotient ring by diagonal invariants, J. Algebraic Combin., 3 (1994), 17--76.


\bibitem{Loehr-higher-qtCatalan}
Nicholas A. Loehr, Conjectured statistics for the higher $q,t$-Catalan sequences, Electron. J. Combin., 12 (2005)
research paper R9; 54 pages (electronic).

\bibitem{Loehr-Warrington}
 Nicholas A. Loehr and Gregory S. Warrington, A continuous family of partition statistics equidistributed with
length, J. Combin. Theory Ser. A, 116 (2009), 379--403.


\bibitem{Gorsky-Mazin2}
 E. Gorsky, M. Mazin, Compactified Jacobians and $q,t$-Catalan Numbers II, J. Algebraic Combin., 39 (2014), 153--186.

\bibitem{Xin-Rank-Complement}
Guoce Xin, Rank complement of rational Dyck paths and conjugation of $(m,n)$-core partitions, preprint:
arXiv:1504.02075.
\end{thebibliography}
\end{document}